\documentclass[12pt,leqno]{amsart}
\usepackage[latin9]{inputenc}
\usepackage{verbatim}
\usepackage{amsthm}
\usepackage{amstext}
\usepackage{amssymb}

\makeatletter
\numberwithin{equation}{section}
\numberwithin{figure}{section}

\usepackage{amsthm}\usepackage{amscd}\usepackage{bm}%\usepackage{graphicx,psfrag,wrapfig}
\theoremstyle{plain}
\newtheorem{theorem}{Theorem}[section]
\newtheorem{lemma}[theorem]{Lemma}
\newtheorem{conjecture}[theorem]{Conjecture}
\newtheorem{corollary}[theorem]{Corollary}
\newtheorem{proposition}[theorem]{Proposition}

\theoremstyle{definition}
\newtheorem{definition}[theorem]{Definition}

\newtheorem{example}[theorem]{Example}

\newtheorem{problem}[theorem]{Problem}

\newcommand{\mdim}{\mathrm{mdim}}
\newcommand{\Diam}{\mathrm{diam}}
\newcommand{\widim}{\mathrm{Widim}}
\newcommand{\dist}{\mathrm{dist}}
\newcommand{\supp}{\mathrm{supp}}
\newcommand{\norm}[1]{\left|\!\left|#1\right|\!\right|}
\makeatother

\begin{document}

\title[Mean dimension and a sharp embedding theorem]{Mean dimension and a sharp embedding theorem: extensions of aperiodic subshifts}

\author{Yonatan Gutman \& Masaki Tsukamoto}



\date{\today}

\maketitle

\begin{abstract}
We show that if $(X,T)$ is an extension of an aperiodic subshift (a subsystem
of $\left(\{1,2,\dots,l\}^{\mathbb{Z}},\mathrm{shift}\right)$ for some $l\in\mathbb{N}$) and
has mean dimension $mdim(X,T)<\frac{D}{2}$ $(D\in \mathbb{N}$), then it embeds equivariantly
in $\left(\left([0,1]^{D}\right)^{\mathbb{Z}},\mathrm{shift}\right)$.
The result is sharp. If $(X,T)$ is an extension of an aperiodic zero-dimensional system then it embeds equivariantly
in $\left(\left([0,1]^{D+1}\right)^{\mathbb{Z}},\mathrm{shift}\right)$.
\end{abstract}

\section{Introduction}

\label{section: introduction} Let $(X,T)$ be a topological dynamical
system (t.d.s),i.e., $X$ is a compact metric space and $T:X\to X$
is a homeomorphism. Let $D$ be a positive integer, and let $[0,1]^{D}$
be the $D$-dimensional unit cube. Let $\left(\left([0,1]^{D}\right)^{\mathbb{Z}},\sigma\right)$
be the shift on $[0,1]^{D}$ i.e., $\sigma(x)_{n}\triangleq x_{n+1}$
for $x=(x_{n})_{n\in\mathbb{Z}}\in\left([0,1]^{D}\right)^{\mathbb{Z}}$.
We study the following problem in this paper: \begin{problem}\label{main problem}
When does there exist an embedding $\phi:(X,T)\to\left(\left([0,1]^{D}\right)^{\mathbb{Z}},\sigma\right)$?
Here $\phi$ is called an (equivariant) embedding if $\phi$ is a
topological embedding and $\phi T=\sigma\phi$. \end{problem} Jaworski
\cite{J74} proved that if $(X,T)$ is a finite dimensional system
and is aperiodic (i.e., it has no periodic points) then $(X,T)$ can
be embedded into the system $\left([0,1]^{\mathbb{Z}},\sigma\right)$
(See also Auslander \cite[Chapter 13, Theorem 9]{A}). Clearly if
a system has periodic points, this may constitute an obstruction to
embedding into $\left(\left([0,1]^{D}\right)^{\mathbb{Z}},\sigma\right)$.
To quantify this introduce an invariant called \textit{periodic dimension
$perdim(X,T)=\sup_{n\in\mathbb{N}}\frac{dim(P_{n})}{n}$,} where $P_{n}$
denotes the set of points of period $\leq n$. One readily checks
that a necessary condition for embedding into $\left(\left([0,1]^{D}\right)^{\mathbb{Z}},\sigma\right)$
is $perdim(X,T)\leq D$. In \cite{Gut12a} it is shown that if $(X,T)$
is a finite dimensional and $perdim(X,T)<\frac{D}{2}$, then $(X,T)$
can be embedded into $\left(\left([0,1]^{D}\right)^{\mathbb{Z}},\sigma\right)$.

For infinite dimensional systems the situation is very different.
Lindenstrauss-Weiss \cite[Propositon 3.3]{LW} constructed a minimal
infinite dimensional system which cannot be embedded into the system
$\left([0,1]^{\mathbb{Z}},\sigma\right)$ by using the theory of mean
dimension.

Mean dimension (denoted by $\mdim(X,T)$) is a topological invariant
of dynamical systems introduced by Gromov \cite{G} and systematically
investigated in \cite{LW}. The explicit formula appears at the end
of the Preliminaries Section. The mean dimension of the system $\left(\left([0,1]^{D}\right)^{\mathbb{Z}},\sigma\right)$
is equal to $D$. If a system $(X,T)$ can be embedded into another
system $(X',T')$, then $\mdim(X,T)\leq\mdim(X',T')$. So if $(X,T)$
can be embedded into $\left(\left([0,1]^{D}\right)^{\mathbb{Z}},\sigma\right)$
then we must have $\mdim(X,T)\leq D$. Hence the mean dimension is
also an obstruction to embedding a system into $\left(\left([0,1]^{D}\right)^{\mathbb{Z}},\sigma\right)$.

Surprisingly Lindenstrauss \cite[Theorem 5.1]{L99} proved the following
partial converse: \begin{theorem}\label{theorem of Lindenstrauss}
If $(X,T)$ is an extension of an aperiodic minimal system and satisfies
$\mdim(X,T)<D/36$ then $(X,T)$ can be embedded into the system $\left(\left([0,1]^{D}\right)^{\mathbb{Z}},\sigma\right)$.
\end{theorem}

In \cite{Gut12b} it is shown one can replace the assumption of an
aperiodic minimal system with the assumption of an aperiodic factor
with a countable number of minimal subsystems or an aperiodic finite
dimensional factor. In \cite{Gut12a} the following theorem is proven:
\begin{theorem} If $(X,T)$ is an extension of an aperiodic finite dimensional
system and satisfies $\mdim(X,T)<D/16$ then $(X,T)$ can be embedded
into the system $\left(\left([0,1]^{D+1}\right)^{\mathbb{Z}},\sigma\right)$.
\end{theorem} It is clear that both mean dimension and periodic dimension
are obstructions for embedding. Recently Lindenstrauss and Tsukamoto
conjectured that these are the only obstructions. Sacrificially Conjecture
1.2 of \cite{LT12} states: \begin{conjecture} If $mdim(X,T)<\frac{D}{2}$
and $perdim(X,T)<\frac{D}{2}$ then $(X,T)$ can be embedded into
the system $\left(\left([0,1]^{D}\right)^{\mathbb{Z}},\sigma\right)$.
\end{conjecture} Our goal in this paper is to establish this conjecture
in a special (infinite-dimensional aperiodic) case and show this result
is sharp. As mentioned before \cite{Gut12a} contains a proof that
conjecture is true if $X$ is finite dimensional.

Throughout this paper we use the following notation: For a map $f:X\to[0,1]^{D}$
we define $I_{f}:X\to\left([0,1]^{D}\right)^{\mathbb{Z}}$ by $I_{f}(x)\triangleq(f(T^{n}x))_{n\in\mathbb{Z}}$.
Our main result is the following: \begin{theorem} \label{thm: main theorem}
Let $D$ be a positive integer. Let $(Z,S)$ be an aperiodic zero
dimensional system, and let $\pi:(X,T)\to(Z,S)$ be an extension of
$(Z,S)$. If $\mdim(X,T)<D/2$, then there exists a continuous map
$f:X\to[0,1]^{D}$ such that
\[
(I_{f},\pi):X\to\left([0,1]^{D}\right)^{\mathbb{Z}}\times Z,\quad x\mapsto(I_{f}(x),\pi(x)),
\]
is an embedding. Indeed such continuous maps $f:X\to[0,1]^{D}$ form
a comeagre subset of $C(X,[0,1]^{D})$ (the space of continuous maps
from $X$ to $[0,1]^{D}$). \end{theorem} \begin{example}\label{example: extension of zero dimensional system}
Let $(X,T)$ be an arbitrary dynamical system, and let $(Z,S)$ be
a zero dimensional system. As $mdim(X\times Z,T\times S)\leq mdim(X,T)+mdim(Z,S)$
(\cite[Proposition 2.8]{LW}), then the product system $(X,T)\times(Z,S)$
is an extension of $(Z,S)$ whose mean dimension is equal to the mean
dimension of $(X,T)$. \end{example}

Theorem \ref{thm: main theorem} enables us to give a partial solution
to Problem \ref{main problem} by using:
 \begin{corollary} \label{cor: general case}
Let $D$ be a positive integer, and let $(X,T)$ be an extension of
an aperiodic zero dimensional system. If $\mdim(X,T)<D/2$, then $(X,T)$
can be embedded into $\left(\left([0,1]^{D+1}\right)^{\mathbb{Z}},\sigma\right)$.
\end{corollary}
\begin{proof}
If $(Z,S)$ is a zero dimensional
system, then $(Z,S)$ can be embedded into $\left([0,1]^{\mathbb{Z}},\sigma\right)$.
Indeed there exists a topological embedding $g:Z\to[0,1]$. Hence
$I_{g}:(Z,S)\to\left([0,1]^{\mathbb{Z}},\sigma\right)$ gives an embedding
of the system. Therefore the product system $\left(\left([0,1]^{D}\right)^{\mathbb{Z}},\sigma\right)\times(Z,S)$
can be embedded into the system $\left(\left([0,1]^{D+1}\right)^{\mathbb{Z}},\sigma\right)$.
Corollary \ref{cor: general case} now follows from Theorem \ref{thm: main theorem}.
\end{proof}
In the context of the previous corollary, it should be noted that for $\mathbb{Z}^{k}$-extensions of aperiodic zero dimensional systems $(\mathbb{Z}^{k},X)$, it was shown in \cite{G11} that  there exist constants $C(k)>0$, so that $(\mathbb{Z}^{k},X)$ embeds in $\left(\left([0,1]^{\lfloor C(k)mdim(\mathbb{Z}^{k},X)\rfloor+1}\right)^{\mathbb{Z}^{k}},\mathbb{Z}^{k}-\textrm{shift}\right)$.
 
If we assume a stronger condition on the factor $(Z,S)$,
then we can get a sharp result as follows. Let $l$ be a positive
integer, and let $\left(\{1,2,\dots,l\}^{\mathbb{Z}},\mathrm{shift}\right)$
be the shift on the alphabet $\{1,2,\dots,l\}$. A subsystem (closed shift-invariant subset) of such a t.d.s is called a \emph{subshift} (see \cite{W04}).
 \begin{corollary}\label{cor: optimal case}
Let $D$ be a positive integer, and let $(X,T)$ be an extension of
an aperiodic subsystem of $\left(\{1,2,\dots,l\}^{\mathbb{Z}},\mathrm{shift}\right)$.
If $\mdim(X,T)<D/2$, then $(X,T)$ can be embedded into $\left(\left([0,1]^{D}\right)^{\mathbb{Z}},\sigma\right)$.
\end{corollary} \begin{proof} Suppose that $(Z,S)$ is an aperiodic
subsystem of $\left(\{1,2,\dots,l\}^{\mathbb{Z}},\mathrm{shift}\right)$
and that $(X,T)$ is an extension of $(Z,S)$. Since $Z$ is zero
dimensional, Theorem \ref{thm: main theorem} implies that $(X,T)$
can be embedded into the system $\left(\left([0,1]^{D}\right)^{\mathbb{Z}},\sigma\right)\times(Z,S)$.
From the assumption the latter is a subsystem of
\begin{equation}
\left(\left([0,1]^{D}\times\{1,2,\dots,l\}\right)^{\mathbb{Z}},\mathrm{shift}\right).\label{eq: intermidiate system}
\end{equation}
The space $[0,1]^{D}\times\{1,2,\dots,l\}$ can be topologically embedded
into $[0,1]^{D}$. Thus the above system (\ref{eq: intermidiate system})
can be embedded into $\left(\left([0,1]^{D}\right)^{\mathbb{Z}},\sigma\right)$.
\end{proof} Corollary \ref{cor: optimal case} is analogous to the
following classical result in dimension theory (\cite[p. 56, Theorem V 2]{HW41}):
If a compact metric space $X$ satisfies $\dim X<D/2$, then $X$
can be topologically embedded into $[0,1]^{D}$. The following proposition
shows that the condition $\mdim(X,T)<D/2$ in Corollary \ref{cor: optimal case}
is optimal. \begin{proposition}\label{prop: example showing optimality}
Let $(Z,S)$ be a zero dimensional system. For any positive integer
$D$, there exists an extension $(X,T)$ of $(Z,S)$ such that $\mdim(X,T)=D/2$
and $(X,T)$ cannot be embedded into $\left(\left([0,1]^{D}\right)^{\mathbb{Z}},\sigma\right)$.
\end{proposition} Indeed this proposition is a corollary of a stronger
result (Proposition \ref{prop: non-embeddability}) given in Section
\ref{section: proof of proposition showing optimality}.

\section{Preliminaries}

\label{section: preliminaries} In this section we present some basic
facts on mean dimension theory. For more details, see Gromov \cite{G}
and Lindenstrauss-Weiss \cite{LW}. Let $(X,d)$ be a compact metric
space, $Y$ be a topological space and $f:X\to Y$ be a continuous
map. $f$ is called an $\varepsilon$-embedding for $\varepsilon>0$
if $\Diam(f^{-1}(y))<\varepsilon$ for every $y\in Y$. We define
$\widim_{\varepsilon}(X,d)$ as the minimum of $n\geq0$ such that
there exist a $n$-dimensional polytope $P$ and an $\varepsilon$-embedding
$f:X\to P$. Equivalently, $\widim_{\varepsilon}(X,d)$ is the minimum
$n\geq0$ such that there exists an open covering $\{U_{i}\}_{i=1}^{N}$
of $X$ satisfying $\Diam(U_{i})<\varepsilon$ $(1\leq i\leq N)$
and $U_{i_{1}}\cap U_{i_{2}}\cap\dots\cap U_{i_{n+2}}=\emptyset$
for all $1\leq i_{1}<i_{2}<\dots<i_{n+2}\leq N$ (i.e. the order of
$\{U_{i}\}_{i=1}^{N}$ is at most $n$). $\widim_{\varepsilon}(X,d)$
is monotone non-decreasing as $\varepsilon\to0$.

The following is a key property of $\widim_{\varepsilon}(X,d)$ in
the study of our embedding problem: \begin{lemma} \label{lemma: existence of good approxiamtion}

Let $(X,d)$ be a compact metric space, and let $f:X\to[0,1]^{M}$
be a continuous map. Suppose that positive numbers $\delta$ and $\varepsilon$
satisfy the following condition:
\[
d(x,y)<\varepsilon\Rightarrow\norm{f(x)-f(y)}_{\infty}<\delta.
\]
(Here $\norm{\cdot}_{\infty}$ is the $\ell^{\infty}$-norm.) Under
this condition, if $\widim_{\varepsilon}(X,d)<M/2$ then there exists
an $\varepsilon$-embedding $g:X\to[0,1]^{M}$ satisfying $\sup_{x\in X}\norm{f(x)-g(x)}_{\infty}<\delta$.
\end{lemma} \begin{proof} Set $a\triangleq\widim_{\varepsilon}(X,d)$.
Let $X=\bigcup_{i=1}^{N}U_{i}$ be an open covering of order $a$
satisfying $\Diam(U_{i})<\varepsilon$. Fix $x_{i}\in U_{i}$, and
let $\{\varphi_{i}\}_{i=1}^{N}$ be a partition of unity satisfying
$\supp\varphi_{i}\subset U_{i}$. From the assumption on $\delta$
and $\varepsilon$, we have
\[
s\triangleq\sup_{x\in\supp\varphi_{i}}\norm{f(x)-f(x_{i})}_{\infty}<\delta.
\]
Since $2a<M$, we can choose $u_{i}\in[0,1]^{M}$ $(1\leq i\leq N)$
satisfying
\[
\norm{u_{i}-f(x_{i})}_{\infty}<\delta-s
\]
and the following condition: If there are $J,K\subset\{1,2,\dots,N\}$
with $|J|,|K|\leq a+1$, $\{\lambda_{i}\}_{i\in J}\in(\mathbb{R}\setminus\{0\})^{|J|}$
and $\{\mu_{i}\}_{i\in K}\in(\mathbb{R}\setminus\{0\})^{|K|}$ such
that
\[
\sum_{i\in J}\lambda_{i}=\sum_{i\in K}\mu_{i}=1,\quad\sum_{i\in J}\lambda_{i}u_{i}=\sum_{i\in K}\mu_{i}u_{i},
\]
then $J=K$ and $\lambda_{i}=\mu_{i}$ for all $i\in J=K$. (The existence
of such $\{u_{i}\}_{i=1}^{N}$ follows from the fact that one can
choose almost surely (w.r.t Lebesgue measure) in $([0,1]^{M})^{N}$,
$\{u_{i}\}_{i=1}^{N}$ to be affinely independent, see \cite[Appendix]{Gut12a}).

We define $g:X\to[0,1]^{M}$ by setting $g(x)\triangleq\sum_{i=1}^{N}\varphi_{i}(x)u_{i}$.
We have $g(x)-f(x)=\sum_{i}\varphi_{i}(x)(u_{i}-f(x))$. If $\varphi_{i}(x)\neq0$
then from the definition of $u_{i}$
\[
\norm{u_{i}-f(x)}_{\infty}\leq\norm{u_{i}-f(x_{i})}_{\infty}+\norm{f(x_{i})-f(x)}_{\infty}<\delta.
\]
Hence $\norm{g(x)-f(x)}_{\infty}<\delta$ for all $x\in X$.

If $g(x)=g(y)$ for some $x,y\in X$, then the choice of $\{u_{i}\}$
implies that there exists $1\leq i\leq N$ satisfying $\varphi_{i}(x)=\varphi_{i}(y)>0$.
Hence $x,y\in U_{i}$. Then $d(x,y)\leq\Diam(U_{i})<\varepsilon$.
This shows that $g$ is an $\varepsilon$-embedding.

\end{proof}

Suppose that $T:X\to X$ is a homeomorphism. For two integers $a<b$
we define a new distance $d_{a}^{b}$ by setting $d_{a}^{b}(x,y)\triangleq\max_{a\leq i\leq b}d(T^{i}x,T^{i}y)$.
We define the mean dimension $\mdim(X,T)$ by
\[
\mdim(X,T)\triangleq\sup_{\varepsilon>0}\left(\lim_{n\to+\infty}\frac{\widim_{\varepsilon}(X,d_{0}^{n-1})}{n}\right).
\]
This limit always exists. (The limit value can be $\infty$.) %
\begin{comment}
Later we will need the following fact: For any integer $a$,
\begin{equation}
\mdim(X,T)=\sup_{\varepsilon>0}\left(\lim_{n\to+\infty}\frac{\widim_{\varepsilon}(X,d_{a}^{n-1})}{n}\right).\label{eq: modification of def of mean dimension}
\end{equation}
\end{comment}

\section{Proof of Theorem \ref{thm: main theorem}}

\label{section: proof of the main theorem} In this section we prove
Theorem \ref{thm: main theorem}. The basic structure of the proof
is the same as Lindenstrauss \cite[Section 5]{L99}. Lindenstrauss
proved Theorem \ref{theorem of Lindenstrauss} by using a Rohlin-tower
like lemma \cite[Lemma 3.3]{L99}. The proof of this tower lemma uses
the assumption that $(X,T)$ is an extension of an infinite minimal
system. Here we assume that $(X,T)$ is an extension of an aperiodic
zero dimensional system. This assumption implies a much stronger tower
lemma (Lemma \ref{lemma: tower lemma}), and thus we are able to prove
a sharp result for the embeddedability of $(X,T)$. Throughout this
section, $(X,d)$ is a compact metric space and $T:X\to X$ is a homeomorphism
such that there exist an aperiodic zero dimensional system $(Z,S)$
and a factor map $\pi:(X,T)\to(Z,S)$. Moreover we assume that $\mdim(X,T)<D/2$
for a positive integer $D$. For simplicity of the notation, we set
$K\triangleq[0,1]^{D}$.

For $\eta>0$, we define a subset $A(\eta)$ of $C(X,K)$ by
\[
A(\eta)\triangleq\{f|\,\text{\ensuremath{(I_{f},\pi):X\to K^{\mathbb{Z}}\times Z}is an \ensuremath{\eta}-embedding with respect to \ensuremath{d}.}\}.
\]
It is easy to see that $A(\eta)$ is an open set. We want to prove
the following proposition: \begin{proposition} \label{prop: main proposition}
For any continuous $f:X\to K$, $\delta>0$ and $\eta>0$ there exists
a continuous map $g:X\to K$ such that

\noindent (i) $\norm{f(x)-g(x)}_{\infty}<\delta$ for all $x\in X$.

\noindent (ii) $(I_{g},\pi):X\to K^{\mathbb{Z}}\times Z$ is an $\eta$-embedding
with respect to the distance $d$. \end{proposition} Assume this
proposition. Then $A(\eta)$ is dense in $C(X,K)$. The Baire Category
Theorem implies that
\[
\bigcap_{n=1}^{\infty}A(1/n)
\]
is comeagre and hence dense. Then every $f$ in this set gives a desired
solution. So the main problem is to prove Proposition \ref{prop: main proposition}.

\subsection{Representation lemma}

\label{subsection: Tower lemma} Recall that $(Z,S)$ is an aperiodic
zero dimensional system, and that the zero dimensionality implies
that clopen subsets (closed and open subsets) form a basis of the
topology of $Z$.

We start by a representation lemma for $(Z,S)$. As it similar to
the well known concept of a \textit{special automorphism} in measured
dynamics (see \cite{EOM}) we call such a representation a \textit{special
topological dynamical system.} \begin{definition} Let $B$ be a zero dimensional
compact metric space. Let $T_{B}:B\rightarrow B$ be an automorphism.
Let $h:B\rightarrow\mathbb{Z}_{\geq0}$ be continuous. Define $Q=\{(b,z)|\ b\in B,\ 0\leq z<h(b)\}\subset B\times\mathbb{Z}_{\geq0}$.
As $h$ is bounded this is a compact metric zero dimensional space.
Define $T:Q\rightarrow Q$ by:
\[
T(b,z)=\begin{cases}
(b,z+1) & \quad z+1<h(b)\\
(T_{B}(b),0) & \quad z+1=h(b)
\end{cases}
\]
Clearly $T$ is an automorphism. $B$ is referred to as the \textit{base
}of $Q$. $h:B\rightarrow\mathbb{Z}_{\geq0}$ is referred to as the
\textit{roof function} of $Q$. $(Q,T)$ is referred to as a \textit{special
topological dynamical system } and it is said to be \textit{induced}
by $(B,T_{B},h)$.

\end{definition}

\begin{lemma}\label{lemma: tower lemma}

For any positive integer $N$ there exists a clopen set $B\subset Z$,
a continuous function $h:B\rightarrow R_{N}\triangleq\{N+1,\ldots,2N+1\}$
and a continuous automorphism $T_{B}:B\rightarrow B$ so that $(Z,S)$
is isomorphic (as a dynamical system) to a special t.d.s $(Q_{N},T')$
induced by $(B,T_{B},h)$. We denote this isomorphism by $(b,n):Z\rightarrow Q_{N}\subset B\times\mathbb{Z}_{\geq0}$.
\end{lemma}

\noindent \textbf{Notation: }We say such $B,T_{B},h$ are \textit{associated}
with $N$.
\begin{proof}
There exist clopen sets $U_{1},\dots,U_{m}$ such that $X=U_{1}\cup\dots\cup U_{m}$
and $U_{i}\cap S^{k}U_{i}=\emptyset$ for all $1\leq i\leq m$ and
$1\leq k\leq N$. We define clopen sets $V_{l}$ $(1\leq l\leq m)$
by $V_{1}\triangleq U_{1}$ and $V_{l+1}\triangleq V_{l}\cup\left(U_{l+1}\setminus\bigcup_{i=-N}^{N}S^{i}V_{l}\right)$.
We can directly check the following two conditions by using the induction
on $l=1,\dots,m$:
\[
V_{l}\cap S^{k}V_{l}=\emptyset\quad(1\leq l\leq m,1\leq k\leq N),
\]
\[
U_{1}\cup\dots\cup U_{l}\subset\bigcup_{k=-N}^{N}S^{k}V_{l}\quad(1\leq l\leq m).
\]
Then the clopen set $B\triangleq V_{m}$ satisfies $B\cap S^{k}B=\emptyset$
$(1\leq k\leq N)$ and $X=\bigcup_{k=1}^{2N+1}S^{k}B$. Define the
continuous function $h:B\rightarrow R_{N}\triangleq\{N+1,\ldots,2N+1\}$
by $h(b)=\min\{n\geq1|\, S^{n}b\in B\}$. Clearly $N<h(b)\leq2N+1$.
Define $T_{B}:B\rightarrow B$ by $T_{B}b=S^{h(b)}b$. $T_{B}$ is
clearly continuous as $h$ is locally constant. $T_{B}$ is invertible
as $T_{B}^{-1}b=S^{-n(b)}b$ where $n(b)=\min\{l\geq1|\, S^{-l}b\in B\}$
and it clearly holds $h(T_{B}^{-1}b)=n(b)$. Notice $n:X\rightarrow\{1,\ldots,2N+1\}$
extends continuously to all of $X$ by the same formula. Let $(Q_{N},T')$
be the a special t.d.s induced by $(B,T_{B},h)$. Define $b:Z\rightarrow B$
by $b(z)=S^{-n(z)}z$. Note $(b,n):(Z,S)\rightarrow(Q_{N},T')$ is
indeed an equivariant isomorphism.
\end{proof}

\subsection{Proof of Proposition \ref{prop: main proposition}}

\label{subsection: proof of main proposition}

Throughout this subsection we fix a continuous function $f:X\to K$
(recall: $K=[0,1]^{D}$) and positive numbers $\delta$ and $\eta$.
Fix $0<\varepsilon<\eta$ such that
\begin{equation}
d(x,y)<\varepsilon\Rightarrow\norm{f(x)-f(y)}_{\infty}<\delta.\label{eq: fix varepsilon}
\end{equation}
Since $\mdim(X,T)<D/2$, we can take a positive integer $L$ such
that for every $k\geq L$
\begin{equation}
\frac{\widim_{\varepsilon}(X,d_{0}^{k-1})}{k}<\frac{D}{2}.\label{eq: fix L}
\end{equation}
\begin{comment}
From (\ref{eq: modification of def of mean dimension}) we can take
a positive integer $M>L$ satisfying
\begin{equation}
\frac{\widim_{\varepsilon}(X,d_{-L}^{M-1})}{M}<\frac{D}{2}.\label{eq: fix M}
\end{equation}
Throughout this subsection we fix the above $\varepsilon$, $L$ and
$M$.
\end{comment}
Applying Lemma \ref{lemma: tower lemma} to $(Z,S)$, let $Q=Q_{L}(B,T_{B},h)$
be the special t.d.s induced by $(B,T_{B},h)$ associated with $L$.
Let $B_{k}=(h\circ\pi)^{-1}\{k\}$, $k\in R_{L}$. These are clopen
sets. We consider the function $(I_{f}|_{0}^{k-1})|_{B_{k}}(x)\triangleq(f(x),f(Tx),\ldots f(T^{k-1}x))\in K^{k}$
on the metric space $(B_{k},d_{0}^{k-1})$. Notice that by (\ref{eq: fix varepsilon})
$d_{0}^{k-1}(x,y)<\varepsilon$ implies $\norm{I_{f}|_{0}^{k-1}(x)-I_{f}|_{0}^{k-1}(y)}_{\infty}<\delta$.
By Lemma \ref{lemma: existence of good approxiamtion} there is an
$\varepsilon$-embedding $F_{k}:(B_{k},d_{0}^{k-1})\to K^{k}$ satisfying
\begin{equation}
\sup_{x\in B_{k}}\norm{F_{k}(x)-I_{f}(x)|_{0}^{k-1}}_{\infty}<\delta,\label{eq:F_k is close to I_f_0^k-1}
\end{equation}
Define $g:X\rightarrow K$ by $g(T^{l}x)\triangleq F_{k}(x)|_{l}$
for $x\in B_{k}$ and $0\leq l<k$. Clearly $g$ is continuous and
by (\ref{eq:F_k is close to I_f_0^k-1}) $\sup_{x\in X}\norm{g(x)-f(x)}_{\infty}<\delta$.
Now assume $(I_{g}(x),\pi(x))=(I_{g}(y),\pi(y))$. As $\pi(x)=\pi(y)$
we have $b(\pi(x))=b(\pi(y))$ and $n\triangleq n(\pi(x))=n(\pi(y))$.
Conclude there exists $k\in R_{L}$ so that with $0\leq n<k$ and
$T^{-n}x,T^{-n}y\in B_{k}$. From $I_{g}(T^{-n}x)|_{0}^{k-1}=I_{g}(T^{-n}y)|_{0}^{k-1}$
we conclude $F_{k}(T^{-n}x)=F_{k}(T^{-n}y)$. As $F_{k}:(B_{k},d_{0}^{k-1})\to K^{k}$
is an $\epsilon$-embedding we conclude $d_{0}^{k-1}(T^{-n}x,T^{-n}y)<\epsilon$
which implies $d(x,y)<\epsilon$. $\hfill\square$

\section{Proof of Proposition \ref{prop: example showing optimality}}

\label{section: proof of proposition showing optimality}

The argument in this section is a slight modification of the argument
of Lindenstrauss-Tsukamoto \cite[Lemma 3.3]{LT12}. Let $Y$ be the
triod graph (the graph of the shape ``$Y$''). Let $d_{Y}$ be the
graph distance on $Y$ where all three edges have length $1$. Consider
$Y^{n}$ with the distance $d_{\ell^{\infty}}(x,y)\triangleq\max_{1\leq i\leq n}d_{Y}(x,y)$.
The following fact is proved in \cite[Proposition 2.5]{LT12} by using
the method of Matou\^{s}ek \cite{M03}: \begin{lemma}\label{lemma: property of Y}
For $0<\varepsilon<1$ there does not exist an $\varepsilon$-embedding
from the space $(Y^{n},d_{\ell^{\infty}})$ to $\mathbb{R}^{2n-1}$.
\end{lemma} Let $D$ be a positive integer. In \cite[Section 3]{LT12}
a compact metric space $(X,d)$ with a homeomorphism $T:X\to X$ satisfying
the following two conditions was constructed:

\noindent (i) $\mdim(X,T)=D/2$.

\noindent (ii) There exist sequences of positive integers $\{c_{n}\}_{n\geq1}$
and $\{b_{n}\}_{n\geq1}$ such that
\[
\lim_{n\to\infty}b_{n}=\infty,\quad\lim_{n\to\infty}\left(c_{n}-\frac{Db_{n}}{2}\right)=\infty,
\]
and so that there exists a distance-increasing continuous map from
$(Y^{c_{n}},d_{\ell^{\infty}})$ to $(X,d_{0}^{b_{n}-1})$ for every
$n\geq1$ (\cite[Lemma 3.1]{LT12}). (A map $f:(Y^{c_{n}},d_{\ell^{\infty}})\to(X,d_{0}^{b_{n}-1})$
is said to be distance-increasing if $d_{0}^{b_{n}-1}(f(x),f(y))\geq d_{\ell^{\infty}}(x,y)$.)

Moreover one can assume that $(X,T)$ is minimal. But this is not
used below. (Remark: In the notation of \cite[Section 3]{LT12}, we
have $c_{n}=Db_{n}\prod_{k=1}^{n}(1-p_{k})$.)

Proposition \ref{prop: example showing optimality} follows from the
consideration in Example \ref{example: extension of zero dimensional system}
and the following proposition. (We set $K=[0,1]^{D}$ as in Section
\ref{section: proof of the main theorem}.) \begin{proposition} \label{prop: non-embeddability}
For any dynamical system $(Z,S)$, the product system $(X,T)\times(Z,S)$
cannot be embedded into $(K^{\mathbb{Z}},\sigma)$. \end{proposition}
\begin{proof} Let $d_{Z}$ be the distance on $Z$. We define a distance
on $X\times Z$ by $\dist((x,z),(x',z'))\triangleq\max(d(x,x'),d_{Z}(z,z'))$.
Fix $p\in Z$. The map
\[
X\to X\times Z,\quad x\mapsto(x,p),
\]
gives an isometric embedding from $(X,d_{0}^{n-1})$ to $(X\times Z,\dist_{0}^{n-1})$
for every $n\geq1$. From property (ii) of $(X,T)$, there exists
a distance-increasing continuous map from $(Y^{c_{n}},d_{\ell^{\infty}})$
to $(X\times Z,\dist_{0}^{b_{n}-1})$ for every $n\geq1$. From this
point on one can use the proof of \cite[Section 3]{LT12} with $X$
replaced by $X\times Z$ verbatim.

Suppose that there exists an embedding $f:(X\times Z,T\times S)\to(K^{\mathbb{Z}},\sigma)$.
Let $d'$ be a distance on $K^{\mathbb{Z}}$. There exists $\varepsilon>0$
such that if $x,y\in X\times Z$ satisfy $d'(f(x),f(y))<\varepsilon$
then $\dist(x,y)<1/2$. Then, for any $N\geq1$, ${d'}_{0}^{N-1}(f(x),f(y))<\varepsilon$
implies $\dist_{0}^{N-1}(x,y)<1/2$.

We can take $L=L(\varepsilon)>0$ such that if $x,y\in K^{\mathbb{Z}}$
satisfy $x_{n}=y_{n}$ $(-L\leq n\leq L)$ then $d'(x,y)<\varepsilon$.
This implies that, for any $N\geq1$, if $x,y\in K^{\mathbb{Z}}$
satisfies $x_{n}=y_{n}$ $(-L\leq n\leq N+L-1)$ then we have ${d'}_{0}^{N-1}(x,y)<\varepsilon$.
Let $\pi_{[-L,N+L-1]}:K^{\mathbb{Z}}\to K^{\{-L,-L+1,\dots,N+L-1\}}$,
$x\mapsto(x_{n})_{-L\leq n\leq N+L-1}$, be the projection. Then,
from the above discussions, $\pi_{[-L,N+L-1]}\circ f:(X\times Z,\dist_{0}^{N-1})\to K^{\{-L,-L+1,\dots,N+L-1\}}$
is a $(1/2)$-embedding. Since there exists a distance-increasing
continuous map from $(Y^{c_{n}},d_{\ell^{\infty}})$ to $(X\times Z,\dist_{0}^{b_{n}-1})$,
we conclude that there exists a $(1/2)$-embedding from $(Y^{c_{n}},d_{\ell^{\infty}})$
to $K^{\{-L,-L+1,\dots,b_{n}+L-1\}}=[0,1]^{D(b_{n}+2L)}$ for every
$n\geq1$. Applying Lemma \ref{lemma: property of Y} to this situation,
we get $2c_{n}\leq D(b_{n}+2L)$. Hence
\[
c_{n}-\frac{Db_{n}}{2}\leq LD.
\]
This contradicts the condition (ii): $\lim_{n\to\infty}(c_{n}-Db_{n}/2)=\infty$.
\end{proof}

\bibliographystyle{alpha}
\bibliography{C:/Users/yonatan/Documents/Math/Bib/universal_bib}

\def\cprime{$'$}
\begin{thebibliography}{Gut12b}

\bibitem[Ao]{EOM}
D.V. Anosov~(originator).
\newblock Special automorphism.
\newblock Encyclopedia of Mathematics (online).

\bibitem[Aus88]{A}
Joseph Auslander.
\newblock {\em Minimal flows and their extensions}, volume 153 of {\em
  North-Holland Mathematics Studies}.
\newblock North-Holland Publishing Co., Amsterdam, 1988.
\newblock Notas de Matem\'atica [Mathematical Notes], 122.

\bibitem[Gro99]{G}
Misha Gromov.
\newblock Topological invariants of dynamical systems and spaces of holomorphic
  maps. {I}.
\newblock {\em Math. Phys. Anal. Geom.}, 2(4):323--415, 1999.

\bibitem[Gut11]{G11}
Yonatan Gutman.
\newblock Embedding {$\Bbb Z\sp k$}-actions in cubical shifts and {$\Bbb Z\sp
  k$}-symbolic extensions.
\newblock {\em Ergodic Theory Dynam. Systems}, 31(2):383--403, 2011.

\bibitem[Gut12a]{Gut12b}
Yonatan Gutman.
\newblock Dynamical embedding in cubical shifts and the topological {R}okhlin
  and small boundary properties.
\newblock Preprint, 2012.

\bibitem[Gut12b]{Gut12a}
Yonatan Gutman.
\newblock Mean dimension and {J}aworski-type theorems.
\newblock Preprint, 2012.

\bibitem[HW41]{HW41}
Witold Hurewicz and Henry Wallman.
\newblock {\em Dimension {T}heory}.
\newblock Princeton Mathematical Series, v. 4. Princeton University Press,
  Princeton, N. J., 1941.

\bibitem[Jaw74]{J74}
A.~Jaworski.
\newblock {\em The Kakutani-Beboutov theorem for groups}.
\newblock Ph.D. dissertation. University of Maryland, 1974.

\bibitem[Lin99]{L99}
Elon Lindenstrauss.
\newblock Mean dimension, small entropy factors and an embedding theorem.
\newblock {\em Inst. Hautes \'Etudes Sci. Publ. Math.}, 89(1):227--262 (2000),
  1999.

\bibitem[LT12]{LT12}
Elon Lindenstrauss and Masaki Tsukamoto.
\newblock Mean dimension and an embedding problem: an example.
\newblock Preprint, 2012.

\bibitem[LW00]{LW}
Elon Lindenstrauss and Benjamin Weiss.
\newblock Mean topological dimension.
\newblock {\em Israel J. Math.}, 115:1--24, 2000.

\bibitem[Mat03]{M03}
Ji{\v{r}}{\'{\i}} Matou{\v{s}}ek.
\newblock {\em Using the {B}orsuk-{U}lam theorem}.
\newblock Universitext. Springer-Verlag, Berlin, 2003.
\newblock Lectures on topological methods in combinatorics and geometry,
  Written in cooperation with Anders Bj{\"o}rner and G{\"u}nter M. Ziegler.

\bibitem[Wil04]{W04}
Susan~G. Williams.
\newblock Introduction to symbolic dynamics.
\newblock In {\em Symbolic dynamics and its applications}, volume~60 of {\em
  Proc. Sympos. Appl. Math.}, pages 1--11. Amer. Math. Soc., Providence, RI,
  2004.

\end{thebibliography}

\vspace{0.5cm}

\address{Yonatan Gutman, Institute of Mathematics, Polish Academy of Sciences,
ul. \'{S}niadeckich~8, 00-956 Warszawa, Poland.}

\textit{E-mail address}: \texttt{y.gutman@impan.pl}

\vspace{0.5cm}

\address{Masaki Tsukamoto \endgraf Department of Mathematics, Kyoto University,
Kyoto 606-8502, Japan}

\textit{E-mail address}: \texttt{tukamoto@math.kyoto-u.ac.jp}
\end{document}